\documentclass[nohyperref]{article}

\usepackage{microtype}
\usepackage{graphicx}
\usepackage{subfigure}
\usepackage{booktabs} 

\usepackage{lipsum}
\usepackage{amsfonts}
\usepackage{graphicx}
\usepackage{epstopdf}
\usepackage{algorithmic}
\usepackage{cite}
\usepackage{multirow}
\usepackage{minimal-defs}
\usepackage[shortlabels]{enumitem}
\usepackage{mathtools}
\usepackage{float}
\usepackage{mathrsfs}
\usepackage{bbm}
\usepackage{tabularx}
\usepackage{arydshln}

\usepackage{makecell}

\usepackage{hyperref}

\usepackage[accepted]{icml2022}

\usepackage{amsmath}
\usepackage{amssymb}
\usepackage{mathtools}
\usepackage{amsthm}

\usepackage[capitalize,noabbrev]{cleveref}

\theoremstyle{plain}
\newtheorem{theorem}{Theorem}[section]
\newtheorem{proposition}[theorem]{Proposition}
\newtheorem{lemma}[theorem]{Lemma}
\newtheorem{corollary}[theorem]{Corollary}
\theoremstyle{definition}

\theoremstyle{remark}

\usepackage[textsize=tiny]{todonotes}

\icmltitlerunning{NysADMM}

\usepackage{xcolor}
\definecolor{purple}{rgb}{0.3,0.0,.4}

\newcommand{\nystrom}{Nystr{\"o}m}
\newcommand{\deff}{d_\textup{eff}}

\begin{document}

\twocolumn[
\icmltitle{NysADMM: faster composite convex optimization via low-rank approximation}



\icmlsetsymbol{equal}{*}

\begin{icmlauthorlist}
\icmlauthor{Shipu Zhao}{equal,yyy}
\icmlauthor{Zachary Frangella}{equal,zzz}
\icmlauthor{Madeleine Udell}{zzz}
\end{icmlauthorlist}

\icmlaffiliation{yyy}{Cornell University, Ithaca, NY, USA.}
\icmlaffiliation{zzz}{Stanford University, Stanford, CA, USA}

\icmlcorrespondingauthor{Shipu Zhao}{sz533@cornell.edu}

\icmlkeywords{Machine Learning, ICML}

\vskip 0.3in
]



\printAffiliationsAndNotice{\icmlEqualContribution} 

\begin{abstract}
This paper develops a scalable new algorithm, called NysADMM,
to minimize a smooth convex loss function with a convex regularizer.
NysADMM accelerates the inexact Alternating Direction Method of Multipliers (ADMM)
by constructing a preconditioner for the ADMM subproblem from a randomized low-rank Nystr{\"o}m approximation. 
NysADMM comes with strong theoretical guarantees: 
it solves the ADMM subproblem in a constant number of iterations
when the rank of the Nystr{\"o}m approximation is the effective dimension of the subproblem regularized Gram matrix.
In practice, ranks much smaller than the effective dimension can succeed, so NysADMM uses an adaptive strategy to choose the rank that enjoys analogous guarantees. 
Numerical experiments on real-world datasets demonstrate that
NysADMM can solve important applications, such as the lasso, logistic regression, and support vector machines, in half the time (or less) required by standard solvers. 
The breadth of problems on which NysADMM beats standard solvers is a surprise:
it suggests that ADMM is a dominant paradigm 
for numerical optimization across a wide range of statistical learning problems 
that are usually solved with bespoke methods. 
\end{abstract}

\section{Introduction}
\label{section:intro}
Consider the composite convex optimization problem
\small
\begin{equation}\label{eq1}
	\text{minimize}_{x\in\mathbb{R}^{d}} \ \ell(Ax; b) + r(x).
\end{equation}
\normalsize
We assume that $\ell$ and $r$ are convex and $\ell$ is smooth. 
In machine learning, generally $\ell$ is a loss function, 
$r$ is a regularizer, $A\in \mathbb{R}^{n\times d}$ is a feature matrix, and $b\in \mathbb{R}^{n}$ is the label or response. 
Throughout the paper we assume that a solution to \eqref{eq1} exists.
A canonical example of \eqref{eq1} is the lasso problem, 
\small
\begin{equation}\label{eq2}
	\text{minimize} \ \frac{1}{2}\|Ax - b\|^2_2 + \gamma\|x\|_1,
\end{equation}
\normalsize
where $\ell(Ax; b) = \frac{1}{2}\|Ax - b\|^2_2$ and $r(x) = \gamma\|x\|_1$. 
We discuss more applications of \eqref{eq1} in \Cref{section:apps}.

The alternating directions method of multipliers (ADMM) is a popular algorithm 
to solve optimization problems of the form \eqref{eq1}.
However, when the matrix $A$ is large, 
each iteration of ADMM requires solving a large subproblem.
For example, consider the lasso where the loss $\ell$ is quadratic. At each iteration, ADMM solves a regularized least-squares problem
at a cost of $O(nd^2)$ flops.
On the other hand, it is not necessary to solve each subproblem exactly 
to ensure convergence: ADMM strategies that solve the subproblems inexactly
are called inexact ADMM, 
and can be shown to converge when 
the sequence of errors is summable \cite{eckstein1992douglas}.
Unfortunately, it can be challenging even to satisfy this relaxed criterion.
Consider again the lasso problem.
At each iteration, inexact ADMM
solves the regularized least-squares subproblem \eqref{ADMMSubProb} approximately,
for example, using the iterative method of conjugate gradients (CG). We call this method inexact ADMM with CG.
The number of CG iterations required to achieve accuracy $\epsilon$ 
increases with the square root of the 
condition number $\kappa_2$ of the regularized Hessian, $O\left(\sqrt{\kappa_{2}}\log(\frac{\kappa_{2}}{\epsilon})\right)$. 
Alas, the condition number of large-scale data matrices is generally high,
and later iterations of inexact ADMM require high accuracy,
so inexact ADMM with CG still converges too slowly to be practical.

In this work we show how to speed up inexact ADMM using preconditioned conjugate gradients (PCG) as a subproblem solver. 
We precondition with randomized Nystr{\"o}m preconditioning \cite{frangella2021randomized}, a technique inspired by recent developments in randomized numerical linear algebra (RandNLA).
We call the resulting algorithm
NysADMM (``nice ADMM''):
inexact ADMM with PCG using randomized Nystr{\"o}m preconditioning.
The Nystr{\"o}m preconditioner reduces the number of iterations required to solve the subproblem to $\epsilon$-accuracy to $O\left(\log(\frac{1}{\epsilon})\right)$,
independent of the condition number. 
For non-quadratic loss functions, NysADMM uses linearized inexact ADMM and accelerates the linear subproblem solve similarly.



\subsection{Contributions}
\label{subsection:contributions}

\begin{enumerate}
    \item We provide a general algorithmic framework for solving large scale lasso, $\ell_1$-regularized logistic regression, and SVM problems. 
    \item Our theory shows that at each iteration only a constant number of matrix vector products (matvecs) are required to solve the ADMM subproblem, provided we have constructed the preconditioner appropriately. 
    If the loss function is quadratic, only a constant number of matvecs are required to achieve convergence.  
    \item We develop a practical adaptive algorithm that increases the rank until the conditions of our theory are met, which ensures the theoretical benefits of the method can be realized in practice.
    \item Even a preconditioner with lower rank often succeeds in speeding up inexact ADMM with PCG. 
    Our analysis is also able to explain this phenomenon.
    \item Our algorithm beats standard solvers such as glmnet, SAGA, and LIBSVM on large dense problems like lasso, logistic regression, and kernalized SVMs: it yields equally accurate solutions and often runs 2--4 times faster. 
\end{enumerate}

\subsection{Related work}
\label{subsection:related-work} 
Our work relies on recent advancements in RandNLA for solving regularized least squares problems $(A^TA + \mu I)x = A^Tb$ for $x$, 
given a design matrix $A\in \mathbb{R}^{n\times d}$, righthand side $b \in \mathbb{R}^n$, and regularization $\mu \in \mathbb{R}$, using a \emph{sketch} of the design matrix $A$ \cite{lacotte2020effective}. 
NysADMM adapts the randomized Nystr{\"o}m preconditioner of \cite{frangella2021randomized}.
These algorithms begin by forming a \emph{sketch} $Y = A\Omega$ of $A$ (or $A^T$)
with a random dimension reduction map $\Omega \in \mathbb{R}^{d\times s}$ \cite{martinsson2020randomized,woodruff2014sketching}.
For example, $\Omega$ may be chosen to have iid Gaussian entries.
These algorithms obtain significant computational speedups by using a sketch size $s\ll \textrm{min}\{n,d\}$
and working with the sketch in place of the original matrix to construct 
a preconditioner for the linear system.
\citet{frangella2021randomized} and \citet{lacotte2020effective} show that these randomized preconditioners 
work well when the sketch size grows with the \emph{effective dimension} (\cref{eqn:deff-proof}) of the Gram matrix
(assuming, for \citet{lacotte2020effective} that we have access to a matrix square root). 
As the effective dimension is never larger than $d$ and often significantly smaller, 
these results substantially improve on prior work in randomized preconditioning \cite{meng2014lsrn,rokhlin2008fast} 
that requires a sketch size $s\gtrsim d$.
Many applications require even smaller sketch sizes: for example, for NysADMM, a fixed sketch size $s=50$ suffices
even for extremely large problems.

We are not the first to use RandNLA to accelerate iterative optimization.
\citet{pilanci2017newton, gower2019rsn} both use iterative sketching to accelerate Newton's method, while \citet{chowdhuri2020speeding} use randomized preconditioning to accelerate interior point methods for linear programmming.
The approach taken here is closest in spirit to \cite{chowdhuri2020speeding}, as we also use randomized preconditioning.
However, the preconditioner used in \cite{chowdhuri2020speeding} requires the data matrix to have many more columns than rows, while ours can handle any (sufficiently large) dimensions.

NysADMM can solve many traditional machine learning problems, 
such as lasso, regularized logistic regression, and support vector machines (SVMs).
In contrast, standard solvers for these problems
use a wider variety of convex optimization techniques. 
For example, one popular lasso solver, glmnet \cite{friedman2010regularization},
relies on coordinate descent (CD), 
while solvers for SVMs, such as LIBSVM \cite{libsvm}, 
more often use
sequential minimal optimization \cite{platt1998fast}, 
a kind of pairwise CD on the dual problem. 
For regularized logistic regression, especially for $\ell_1$ regularization, 
stochastic gradient algorithms are most commonly used \cite{schmidt2017minimizing, NIPS2014_ede7e2b6}. 
Other authors propose to solve lasso with ADMM \cite{MAL-016, yue2018implementing}. 
Our work, motivated by the ADMM quadratic programming framework
of \citet{osqp}, 
is the first to 
accelerate ADMM with randomized preconditioning, 
thereby improving on the performance of standard CD or stochastic gradient solvers for each of these important classes of machine learning problems on 
large-scale dense data. 
Unlike \citet{osqp}, our work relies on inexact ADMM and can handle non-quadratic loss functions, which allows NysADMM to solve problems such as regularized logistic regression. 

\subsection{Organization of the paper}
\Cref{section:algo} introduces the NysADMM algorithm and necessary background from RandNLA.
\Cref{section:apps} lists a variety of applied problems 
that can be solved by NysADMM.
\Cref{section:Convergence} states the theoretical guarantees for NysADMM. 
\Cref{section:experiments} compares NysADMM and standard optimization solvers numerically on several applied problems.
\Cref{section:conclusion} summarizes the results of the paper and discusses directions for future work.

\subsection{Notation and preliminaries}
We call a matrix psd if it is positive semidefinite. The notation $a\gtrsim b$ means that $a\geq Cb$ for some absolute constant $C$.
Given a matrix $H$, we denote its spectral norm by $\|H\|$. 
We denote the Moore-Penrose pseudoinverse of a matrix $M$ by $M^{\dagger}$.
For $\rho>0$ and a symmetric psd matrix $H$, we define $H_\rho = H+\rho I$.
We say a positive sequence $\{\varepsilon^{k}\}_{k=1}^{\infty}$ is summable if $\sum_{k=1}^{\infty}\varepsilon^{k}<\infty$.
We denote the Loewner ordering on the cone of symmetric psd matrices by $\preceq$, that is $A\preceq B$ if and only if $B-A$ is psd.

\section{Algorithm}
\label{section:algo}
\subsection{Inexact linearized ADMM}
\label{subsection:inexactlinearizedADMM}
To solve problem \eqref{eq1}, we apply the ADMM framework.
\Cref{alg-admm} shows the standard ADMM updates, 
where the regularizer $r = g+h$ is split into a smooth part $g$ and a nonsmooth part $h$. 
\begin{algorithm}[H]
	\centering
	\caption{ADMM}
	\label{alg-admm}
	\begin{algorithmic}
	\INPUT{feature matrix $A$, response $b$, loss function $\ell$, regularization $g$ and $h$, stepsize $\rho$}
		\REPEAT
		\STATE{$x^{k+1} = \text{argmin}_x \{\ell(Ax; b) + g(x) + \frac{\rho}{2}\|x - z^k + u^k\|^2_2\}$}
		\STATE{$z^{k+1} = \text{argmin}_z \{h(z) + \frac{\rho}{2}\|x^{k+1} - z + u^k\|^2_2\}$}
		\STATE{$u^{k+1} = u^k + x^{k+1} - z^{k+1} $}
		\UNTIL{convergence}
	\OUTPUT{solution $x_\star$ of problem \eqref{eq1}}
	\end{algorithmic}
\end{algorithm} 

In each iteration, two subproblems are solved sequentially to update variables $x$ and $z$. 
The $z$-subproblem often has a closed-form solution. 
For example, if $h(x) = \|x\|_1$, the $z$-subproblem is the soft thresholding, 
and if $h$ is the indicator function of a convex set $\mathcal{C}$, 
the $z$-subproblem is projection onto the set $\mathcal{C}$. 

There is usually no closed-form solution for the $x$-subproblem. 
Instead, it is usually solved inaccurately by an iterative scheme,
especially for large-scale applications. 
To simplify the subproblem, inspired by linearized ADMM, 
we assume $\ell$ and $g$ are twice differentiable
and notice that the $x$ update is close to the minimum of a quadratic function 
given by the Taylor expansion of $\ell$ and $g$ at the current iterate:
\begin{equation}\label{eq4}	
	\begin{aligned}
		& \tilde x^{k+1} = \text{argmin}_x \{ \ell(A \tilde x^k ; b) + (x - \tilde x^k)^TA^T \nabla \ell(A \tilde x^k ; b)  \\
		& + \frac{1}{2}(x - \tilde x^k)^T A^TH^{\ell} (A \tilde x^k ; b) A(x - \tilde x^k) + g( \tilde x^k) \\
		& + (x - \tilde x^k)^T\nabla g( \tilde x^k) + \frac{1}{2}(x - \tilde x^k)^T H^g(\tilde x^k) (x - \tilde x^k) \\ 
		& + \frac{\rho}{2}\|x - \tilde z^k + \tilde u^k\|^2_2\}.
	\end{aligned}
\end{equation}
\normalsize
Here $H^\ell$ and $H^g$ are the Hessian of $\ell$ and $g$ respectively. 
We assume throughout the paper that $H^{\ell}$ and $H^{g}$ are psd matrices, this is a very minor assumption, and is satisfied by all the applications we consider.
The solution to this quadratic minimization may be obtained 
by solving the linear system
\begin{eqnarray}\label{ADMMSubProb}
		& (A^TH^{\ell}(A \tilde x^k ; b)A + H^g( \tilde x^k) + \rho I)x = r^k \\
		\mbox{where} & r^k = \rho \tilde z^k - \rho \tilde u^k + A^TH^{\ell}(A \tilde x^k ; b)A \tilde x^k \\ \nonumber
		& + H^g(\tilde x^k)\tilde x^k - A^T \nabla \ell(A\tilde x^k ; b) - \nabla g(\tilde x^k). \nonumber
\end{eqnarray}  
\normalsize
The inexact ADMM algorithm we propose solves \eqref{ADMMSubProb} approximately at each iteration. 
\begin{algorithm}[H]
	\centering
	\caption{Inexact ADMM}
	\label{fADMM}
	\begin{algorithmic}
	\INPUT{feature matrix $A$, response $b$, loss function $\ell$, regularization $g$ and $h$, stepsize $\rho$, positive summable sequence $\{\varepsilon^k\}_{k = 0}^{\infty}$}
		\REPEAT
		\STATE{find $\tilde x^{k+1}$ that solves \eqref{ADMMSubProb} within tolerance $\varepsilon^k$}
		\STATE{$\tilde z^{k+1} = \text{argmin}_z \{h(z) + \frac{\rho}{2}\|\tilde x^{k+1} - z + \tilde u^k\|^2_2\}$}
		\STATE{$\tilde u^{k+1} = \tilde u^k + \tilde x^{k+1} - \tilde z^{k+1} $}
		\UNTIL convergence
	\OUTPUT{solution $x_\star$ of problem \eqref{eq1}}
	\end{algorithmic}
\end{algorithm} 

For a quadratic loss $\ell$, when $\sum_{k = 0}^{\infty} \varepsilon^k < \infty$ and under various other conditions, if optimization problem \eqref{eq1} has an optimal solution, the $\{\tilde x^k\}_{k = 0}^{\infty}$ sequence generated by \Cref{fADMM} converges to the optimal solution of \eqref{eq1} \cite{eckstein1992douglas,eckstein2016approximate}. 

From \citet{MAL-016}, quantity $r_d^{k+1} = \rho (\tilde z^k - \tilde z^{k+1})$ can be regarded as the dual residual and $r_p^{k+1} = \tilde x^{k+1} - \tilde z^{k+1}$ can be viewed as the primal residual at iteration $k+1$. This suggests that we can terminate the ADMM iterations when the primal and dual residuals become very small. The primal and dual tolerances can be chosen based on an absolute and relative criterion, such as  
\begin{equation*}
    \begin{aligned}
        \|r_p^k\|_2 & \le \epsilon^{\text{abs}} + \epsilon^{\text{rel}} \text{max}\{\|\tilde x^k\|_2, \|\tilde z^k\|_2\} \\
        \|r_d^k\|_2 & \le \epsilon^{\text{abs}} + \epsilon^{\text{rel}} \|\rho \tilde u^k\|_2.
    \end{aligned}
\end{equation*}
The relative criteria $\epsilon^{\text{rel}}$ might be $10^{-3}$ or $10^{-4}$ in practice. The choice of absolute criteria $\epsilon^{\text{abs}}$ depends on the scale of the variable values. More details can be found in \citet{MAL-016}.

\subsection{Randomized Nystr{\"o}m approximation and PCG}
\label{subsection:nystrom}
Nystr{\"o}m approximation constructs a low-rank approximation of 
a symmetric psd matrix $H$.
Let $\Omega \in \mathbb{R}^{d\times s}$ be a test matrix 
(often, random Gaussian \citet{NIPS2017_4558dbb6,frangella2021randomized}) 
with sketch size $s\geq 1$.
The Nystr{\"o}m approximation with respect to $\Omega$ is given by
\begin{equation}
\label{eq:NysAppx}
    H\langle \Omega \rangle = (H\Omega)(\Omega^{T}H\Omega)^{\dagger}(H\Omega)^{T}.
\end{equation}
The Nystr{\"o}m approximation $H\langle \Omega \rangle$ is symmetric, psd, and has rank at most $s$ \cref{NysPropLemma}. 
Naive implementation of the Nystr{\"o}m approximation based on \eqref{eq:NysAppx} is numerically unstable. 
\Cref{Nyssketch} in \Cref{section:AlgNysAppxNysPCG} states a stable procedure to compute a randomized Nystr{\"o}m approximation from \citet{NIPS2017_4558dbb6}. 

\Cref{Nyssketch} returns the randomized Nystr{\"o}m approximation of matrix $H$ in the form of an eigendecomposition: $H_{\text{nys}} = U \hat \Lambda U^T$. Let $\hat \lambda_{s}$ be the $s$th eigenvalue. 
The randomized Nystr{\"o}m preconditioner and its inverse take the form 
\begin{equation}
\label{Nyspreconditioner}
\begin{aligned}
    P & = \frac{1}{\hat \lambda_{s} + \rho} U(\hat \Lambda + \rho I)U^T + (I - UU^T),\\
    P^{-1} & = (\hat \lambda_{s} + \rho) U(\hat \Lambda + \rho I)^{-1}U^T + (I - UU^T) \\
\end{aligned}
\end{equation}
\normalsize
\cite{frangella2021randomized}.
In a slight abuse of terminology, we sometimes refer to the sketch size $s$ as the rank of the Nystr{\"o}m preconditioner. 
We will use the the term sketch size and rank interchangeably throughout the paper. 
The Nystr{\"o}m preconditioner may be applied to vectors in $O(ds)$ time and only requires $O(ds)$ floating point numbers to store. 
The details of how to implement PCG with \eqref{Nyspreconditioner} are provided in \Cref{section:AlgNysAppxNysPCG} in \Cref{Nyspcg}.
We now provide some background on Nystr{\"o}m PCG and motivation for why we have paired it with ADMM.

Nystr{\"o}m PCG improves on standard CG both in theory and in practice
for matrices with a small \emph{effective dimension} \cite{frangella2021randomized},
which we now define.
Given a symmetric psd matrix $H\in \mathbb{R}^{d\times d}$ and regularization $\rho>0$, the \emph{effective dimension} of $H$ is
\begin{equation}
\label{eqn:deff-proof}
    \deff(\rho) = \textrm{tr}(H(H+\rho I)^{-1}).
\end{equation}
The effective dimension may be viewed as smoothed count of the eigenvalues of $H$ greater than or equal to $\rho$. 
We always have $d_{\textrm{eff}}(\rho)\leq d$, and we expect $d_{\textrm{eff}}(\rho)\ll d$ whenever $H$ exhibits spectral decay. 

In machine learning, most feature matrices naturally exhibit polynomial or exponential spectral decay \cite{DerezinskiLLM20}, thus we expect that $\deff \ll d$. 
The randomized Nystr{\"o}m preconditioner in \citet{frangella2021randomized} exploits the smallness of $d_{\textrm{eff}}(\rho)$ to build an highly effective preconditioner. 
\citet{frangella2021randomized} show that if \eqref{Nyspreconditioner} is constructed with a sketch size $s\gtrsim d_{\textrm{eff}}(\rho)$, then the condition number of the preconditioned system is constant with high probability. 
An immediate consequence is that PCG solves the preconditioned system to $\epsilon$-accuracy in $O\left(\log(\frac{1}{\epsilon})\right)$ iterations,
independent of the condition number of $H$.

Observe the Hessian $A^{T}H^{\ell}A+H^g$ in the inexact ADMM subproblem \eqref{ADMMSubProb} is formed from the feature matrix $A$. 
Based on the preceding discussion, we expect the Hessian to exhibit spectral decay and for the effective dimension to be small to moderate in size.
Hence we should expect Nystr{\"o}m PCG to accelerate the solution of \eqref{ADMMSubProb} significantly.

\subsection{NysADMM}
\label{subsection:NysADMM}
Integrating Nystr{\"o}m PCG with inexact ADMM, we obtain NysADMM, presented in \Cref{NysADMM}.
\begin{algorithm}[H]
	\centering
	\caption{NysADMM}
	\label{NysADMM}
	\begin{algorithmic}
	\INPUT{feature matrix $A$, response $b$, loss function $\ell$, regularization $g$ and $h$, stepsize $\rho$, positive summable sequence $\{\varepsilon^k\}_{k = 0}^{\infty}$}
	    \STATE{$[U, \hat \Lambda] = \text{RandNystr{\"o}mApprox}(A^T H^{\ell}A + H^g,s)$} \COMMENT{use \Cref{Nyssketch} in \Cref{section:AlgNysAppxNysPCG}}
		\REPEAT
		\STATE{use Nystr{\"o}m PCG (\Cref{Nyspcg} in \Cref{section:AlgNysAppxNysPCG}) to find $\tilde x^{k+1}$ that solves \eqref{ADMMSubProb} within tolerance $\varepsilon^k$}
		\STATE{$\tilde z^{k+1} = \text{argmin}_z \{h(z) + \frac{\rho}{2}\|\tilde x^{k+1} - z + \tilde u^k\|^2_2\}$}
		\STATE{$\tilde u^{k+1} = \tilde u^k + \tilde x^{k+1} - \tilde z^{k+1} $}
		\UNTIL convergence
	\OUTPUT{solution $x_\star$ of problem \eqref{eq1}}
	\end{algorithmic}
\end{algorithm} 
Our theory for \Cref{NysADMM}, shows that if the sketch size $s\gtrsim d_{\textrm{eff}}(\rho)$, then with high probability subproblem \eqref{ADMMSubProb} will be solved to $\epsilon$-accuracy in $O\left(\log(\frac{1}{\epsilon})\right)$ iterations (\Cref{corr:ADMMSubProb}).
When the loss $\ell$ is quadratic
and the sequence of tolerances $\{\varepsilon^k\}_{k = 0}^{\infty}$ is decreasing with $\sum_{k = 0}^{\infty} \varepsilon^k < \infty$,
NysADMM is guaranteed to converge as $k\rightarrow \infty$ with only a constant number of matvecs per iteration (\Cref{thm:NysADMMConv}).
\Cref{t-ComplexityComparison} compares the complexity of inexact ADMM with CG vs.~NysADMM for $K$ iterations under the hypotheses of \Cref{thm:NysADMMConv}.
NysADMM achieves a significant decrease in runtime over inexact ADMM with CG, as the iteration complexity no longer depends on the condition number $\kappa_2$.

\begin{table}[h]
	\begin{center}
		\footnotesize
		\caption{Complexity comparison, for a quadratic loss with Hessian $H$. Here $T_\text{mv}$ is the time to compute a matrix vector product with $H$, $\kappa_2$ is the condition number of $H$, and $\varepsilon^k$ is the precision of the $k$th subproblem solve \eqref{ADMMSubProb}.} 
		\label{t-ComplexityComparison}
		\begin{tabular}{|c|c|}
			\hline
			\textbf{Method} & \textbf{Complexity}\\
			\hline
			\makecell{Inexact ADMM \\ with CG}
			& $O\left(\sum^{K}_{k=1}T_\textrm{mv}\sqrt{\kappa_2}\log\left(\frac{\kappa_2}{\varepsilon^k}\right)\right)$ \\
			\hline
			NysADMM & \makecell{$O\left(T_{\textrm{mv}}d_{\textup{eff}}(\rho)\right)+$ \\ $\sum_{k=1}^K T_{\textrm{mv}}\left(4 + \bigg\lceil2\log\left(\frac{R}{\varepsilon^k \rho}\right) \bigg\rceil\right)$} \\
			\hline
		\end{tabular}
	\end{center}
\end{table}

\subsection{AdaNysADMM}
\label{section:AdaNysADMM}
Two practical problems remain in realizing the success predicted by the theoretical analysis of \Cref{t-ComplexityComparison}.
These bounds are achieved by selecting the sketch size to be $d_{\textrm{eff}}(\rho)$, but the effective dimension is
1) seldom known in practice, and 
2) often larger than required to achieve good convergence of NysADMM.
Fortunately, a simple adaptive strategy for choosing the sketch size,
inspired by \citet{frangella2021randomized},
can achieve the same guarantees as in \Cref{t-ComplexityComparison}. 
This strategy chooses a tolerance $\epsilon$ and doubles the sketch size $s$ until the empirical condition number $\frac{\hat{\lambda}_s+\rho}{\rho}$ satisfies
\begin{equation}
\label{eq:empCondNum}
    \frac{\hat{\lambda}_s+\rho}{\rho} \leq 1 + \epsilon.
\end{equation}
\Cref{thm:AdaNysADMM} guarantees that \eqref{eq:empCondNum}
 holds when $s \geq d_{\textrm{eff}}(\rho)$ and that when \eqref{eq:empCondNum} holds, the true condition number is on the order of $1+\epsilon$ with high probability. We refer to \eqref{eq:empCondNum} as the empirical condition number as it provides an estimate of the true condition number of the preconditoned system (\Cref{thm:AdaNysADMM}).

Thus, to enjoy the guarantees of \Cref{thm:AdaNysADMM} in practice, we may employ the adaptive version of NysADMM, which we call AdaNysADMM.
We provide the pseudocode for AdaNysADMM in \Cref{alg:AdaNysADMM} in \Cref{section:AlgNysAppxNysPCG}. Furthermore, as we use a Gaussian test matrix, it is possible to construct a larger sketch from a smaller one. 
Hence the total computational work needed by the adaptive strategy is not much larger than if the effective dimension were known in advance. 
Indeed, AdaNysADMM differs from NysADMM only in the construction of the preconditioner. The dominant cost in forming the precondition is computing the sketch is $H\Omega$, which costs $O(T_{\textrm{mv}}d_{\textrm{eff}}(\rho))$. As AdaNysADMM reuses computation, the dominant complexity for constructing the Nystr{\"o}m preconditioner remains $O(T_{\textrm{mv}}d_{\textrm{eff}}(\rho))$.
Consequently, the overall complexity of AdaNysADMM is the same as NysADMM in \Cref{t-ComplexityComparison}.

\section{Applications}
\label{section:apps}
Here we discuss various applications that can be reformulated as instances of \eqref{eq1} and solved by \Cref{NysADMM}.
\subsection{Elastic net}
\label{subsection:elnet}
Elastic net generalizes lasso and ridge regression by adding 
both the $\ell_1$ and $\ell_2$ penalty to the least squares problem:
\small
\begin{equation}\label{eq6}
	\text{minimize} \quad \frac{1}{2}\|Ax - b\|^2_2 + \frac{1}{2}(1 - \gamma)\|x\|^2_2 + \gamma \|x\|_1
\end{equation}
\normalsize
Parameter $\gamma>0$ interpolates between the $\ell_1$ and $\ell_2$ penalties. 
NysADMM applies with $\ell(Ax; b) = \frac{1}{2}\|Ax - b\|^2_2$, $g(x) = \frac{1}{2}(1 - \gamma)\|x\|^2_2$, and $h(x) = \gamma \|x\|_1$. 
The Hessian matrices for $\ell$ and $g$ are $A^TA$ and $(1 - \gamma)I$ respectively. 

\subsection{Regularized logistic regression}
\label{subsection:logisticreg}
Regularized logistic regression minimizes a logistic loss function together with an $\ell_1$ regularizer:
\small
\begin{equation}\label{eq9}
	\text{minimize} \ - \sum_i{(b_i (A x)_i - \text{log}(1 + \text{exp}((Ax)_i)))} +  \gamma \|x\|_1
\end{equation}
\normalsize
NysADMM applies with $\ell(Ax; b) = - \sum_i{(b_i (A x)_i - \text{log}(1 + \text{exp}((Ax)_i)))}$ and 
$h(x) = \gamma \|x\|_1$. 
The inexact ADMM update chooses $x^{k+1}$ to minimize a quadratic approximation of the log-likelihood,
\small
\begin{equation*}
	\text{minimize} \ \frac{1}{2}\sum_i{w^k_i(q^k_i - (A x)_i)^2} +  \frac{\rho}{2}\|x - \tilde z^k + \tilde u^k\|^2_2,
\end{equation*}
\normalsize
where $w^k_i$ and $q^k_i$ depend on the current estimate $\tilde x^k$ as 
\small
\begin{equation*}
	\begin{aligned}
		w^k_i = & \frac{1}{2+\text{exp}(-(A \tilde x^k)_i) +\text{exp}((A \tilde x^k)_i)} \\
		q^k_i = & (A \tilde x^k)_i + \frac{b_i - \frac{1}{1+\text{exp}(-(A \tilde x^k)_i)}}{w^k_i}.
	\end{aligned}
\end{equation*}
\normalsize
Therefore, the solution of the $x$-subproblem can be approximated by solving the linear system 
\small
\begin{equation*}
	(A^T\text{diag}(w^k)A + \rho I)x = \rho \tilde z^k - \rho \tilde u^k + A^T \text{diag}(w^k)q^k. 
\end{equation*}
\normalsize
Here $w^k$ and $q^k$ are the vectors for $w^k_i$ and $q^k_i$. 
The Hessian matrix of $\ell$ is given by $A^T\text{diag}(w^k)A$. 

\subsection{Support vector machine}
\label{subsection:svm}
To reformulate the SVM problem for solution with NysADMM, consider the dual SVM problem
\small
\begin{equation}\label{eq8}
	\begin{aligned}
		\text{minimize} \quad &\frac{1}{2} x^T \text{diag}(b) K \text{diag}(b) x - \mathbf{1}^T x \\
		\text{subject to} \quad &x ^T b = 0 \\
		& 0 \le x \le C. \\
	\end{aligned}
\end{equation}
\normalsize
Variable $x$ is the dual variable, 
$b$ is the label or response, and $C$ is the penalty parameter for misclassification. 
For linear SVM, $K = A^TA$ where $A$ is a feature matrix; and for nonlinear SVM, $K$ is the corresponding kernel matrix.  
The SVM problem can be reformulated as \eqref{eq1} by setting $\ell(Ax; b) = \frac{1}{2} x^T \text{diag}(b) K \text{diag}(b) x$, 
$g(x) = -\mathbf{1}^Tx$, and $h$ is the indicator function for convex constraint set $x ^T b = 0, \ 0 \le x \le C$. 
The Hessian matrix for $\ell$ is $\text{diag}(b) K \text{diag}(b)$. 
 
\section{Convergence analysis}
\label{section:Convergence}
This section provides a convergence analysis for NysADMM. All proofs for the results in this section may be found in \Cref{section:proofs}.
First we show Nystr{\"o}m PCG can solve any quadratic problem in a constant number of iterations. 
\begin{theorem}
\label{thm:NysCondNum}
Let $H$ be a symmetric positive semidefinite matrix, $\rho > 0$ and set $H_{\rho} = H+\rho I$. Suppose we construct the randomized Nystr{\"o}m preconditioner with sketch size $s\geq 8\left(\sqrt{d_{\textup{eff}}(\rho)}+\sqrt{8\log(\frac{16}{\delta}})\right)^{2}$. Then
\begin{equation}
    \kappa_{2}(P^{-1/2}H_{\rho}P^{-1/2}) \leq 8
\end{equation}
with probability at least $1-\delta$.
\end{theorem}
\Cref{thm:NysCondNum} strengthens results in \citet{frangella2021randomized}, which provides sharp expectation bounds on the condition number of the preconditioned system, but gives loose high probability bounds based on Markov's inequality. 
Our result tightens these bounds, showing that Nystr{\"o}m PCG enjoys an exponentially small failure probability.

As an immediate corollary, we can solve \eqref{ADMMSubProb} with a few iterations of PCG using the Nystr{\"o}m preconditioner.
\begin{corollary}
\label{corr:ADMMSubProb}
Instate the hypotheses of \Cref{thm:NysCondNum} and let $\tilde{x}_{\star}$ denote the solution of \eqref{ADMMSubProb}. 
Then with probability at least $1-\delta$, the iterates $\{x_t\}_{t\geq 1}$ produced by Nystr{\"o}m PCG on problem \eqref{ADMMSubProb} satisfy
\begin{equation}
    \frac{\|x_t-\tilde{x}_\star\|_2}{\|\tilde{x}_{\star}\|_2}\leq \left(\frac{1}{2}\right)^{t-4}.
\end{equation}
Thus, after $t\geq \bigg\lceil\frac{\log\left(\frac{16\|\tilde{x}_\star\|_2}{\epsilon}\right)}{\log(2)} \bigg\rceil$ iterations,
\begin{equation}
    \|x_t-\tilde{x}_\star\|_2\leq \epsilon.
\end{equation}
\end{corollary}
\Cref{corr:ADMMSubProb} ensures that we can efficiently solve the sub-problem to the necessary accuracy at each iteration. This result allows us to prove convergence of NysADMM.
\begin{theorem}
\label{thm:NysADMMConv}
Consider the problem in \eqref{eq1} with quadratic loss $\ell(Ax; b) = \frac{1}{2}\|Ax-b\|_2^{2}$ and the smooth part $g$ of regularizer $r$ has constant Hessian. Define initial iterates $\tilde x^0$, $\tilde z^0$ and $\tilde u^0\in \mathbb{R}^{d}$, stepsize $\rho>0$, and summable tolerance sequence $\{\varepsilon^k\}^{\infty}_{k=0}\subset \mathbb{R}_{+}$.
Assume at $k$th ADMM iteration, the norm of the righthand side of the linear system $r^k$ is bounded by constant $R$ for all $k$.
Construct the Nystr{\"o}m preconditioner with sketch size 
\[
s\geq 8\left(\sqrt{d_{\textup{eff}}(\rho)}+\sqrt{8\log\left(\frac{16}{\delta}\right)}\right)^{2}
\]
and solve problem \eqref{eq1} with NysADMM, using 
$T^{k} = 4 + \bigg\lceil2\log\left(\frac{R}{\varepsilon^k \rho}\right) \bigg\rceil$ iterations for PCG at the $k$th ADMM iteration. 
Then with probability at least $1-\delta$,
\begin{enumerate}
    \item For all $k\geq 0$, each iterate $\tilde x^{k+1}$ satisfies 
    \begin{equation}
        \|\tilde x^{k+1}-x^{k+1}\|_2\leq \varepsilon^k,
    \end{equation}
    where $x^{k+1}$ is the exact solution of \eqref{ADMMSubProb}.
    \item As $k\rightarrow \infty$, $\{\tilde x^k\}^{\infty}_{k=0}$ converges to a solution of the primal \eqref{eq1} and $\{\rho \tilde u^k\}^{\infty}_{k=0}$ converges to a solution of the dual problem of \eqref{eq1}.
\end{enumerate}
\end{theorem}
\Cref{thm:NysADMMConv} establishes convergence of NysADMM for a quadratic loss.
The quadratic loss already covers many applications of interest including the lasso, elastic-net, and SVMs. 
We conjecture that a modification of our argument can show that NysADMM converges linearly for any strongly convex loss, but we leave this extension to future work.

The next result makes rigorous the claims made in \Cref{section:AdaNysADMM}: it shows we can determine whether or not we have reached the effective dimension by monitoring the empirical condition number $(\hat{\lambda}_s+\rho)/\rho$.
\begin{theorem}
\label{thm:AdaNysADMM}
Suppose, for some user defined tolerance $\epsilon>0$, 
the sketch size satisfies 
\small
\[
s \geq 8\left(\sqrt{d_{\textup{eff}}\left(\frac{\epsilon\rho}{6}\right)}+\sqrt{8\log\left(\frac{16}{\delta}\right)}\right)^{2}.
\]
\normalsize
Then the empirical condition number of the Nystr{\"o}m preconditioned system $P^{-1/2}H_r P^{-1/2}$ satisfies
\begin{equation}
\frac{\hat{\lambda}_{s}+\rho}{\rho} \leq 1+\frac{\epsilon}{42}.
\end{equation}

Furthermore, with probability at least $1-\delta$,
\begin{equation}
\left|\kappa_2(P^{-1/2}H_\rho P^{-1/2})-\frac{\hat{\lambda}_{s}+\rho}{\rho}\right|\leq \epsilon.
\end{equation}
\end{theorem}
\Cref{thm:AdaNysADMM} shows that once the empirical condition number is sufficiently close to $1$, so too is the condition number of the preconditioned system. 
Hence it is possible to reach the effective dimension by doubling the sketch size of the Nystr{\"o}m approximation until the empirical condition number falls below the desired tolerance. 
\Cref{thm:AdaNysADMM} ensures the true condition number is close to this empirical estimate with high probability. 

\Cref{thm:AdaNysADMM} also helps explain why sketch sizes much smaller than the effective dimension can succeed in practice. 
The point is best illustrated by instantiating an explicit parameter selection in \Cref{thm:AdaNysADMM}, which yields the following corollary. 
\begin{corollary}
\label{corr:SmallRanks}
Instate the hypotheses of \Cref{thm:AdaNysADMM} with $\epsilon = 100$. Then with a sketch size of $s\gtrsim d_{\textup{eff}}(16\rho)$ the following holds
\begin{enumerate}
    \item $(\hat{\lambda}_{s}+\rho)/\rho \leq 1+\frac{100}{42}.$
    \item With probability at least $1-\delta$, 
    \[\bigg|\kappa_{2}(P^{-1/2}H_{\rho}P^{-1/2})-1-\frac{100}{42}\bigg|\leq 100.\]
\end{enumerate}
\end{corollary}

\Cref{corr:SmallRanks} shows that for a coarse tolerance of $\epsilon = 100$, a sketch size of $s\gtrsim d_{\textrm{eff}}(16\rho)$ suffices to ensure that the condition number of $P^{-1/2}H_{\rho}P^{-1/2}$ is no more than around $100$.
Two practical observations cement the importance of this corollary.
First, $d_{\textrm{eff}}(16\rho)$ is often significantly smaller than $d_{\textrm{eff}}(\rho)$, possibly by an order of magnitude or more.
Second, with a condition number around $100$, PCG is likely to converge very quickly. 
In fact, for modest condition numbers, PCG is known to converge much faster in practice than the theory would suggest \cite{trefethen1997numerical}.
It is only when the condition number reaches around $10^3$, that convergence starts to slow.
Thus, \Cref{corr:SmallRanks} helps explain why it is not necessary for the sketch size to equal the effective dimension in order for NysADMM to obtain significant accelerations.

\section{Numerical experiments}
\label{section:experiments}
\begin{table}[h]
\caption{Statistics of experiment datasets.}	
\label{dataset}
\vskip 0.15in
\begin{center}
\begin{small}
\begin{tabular}{lccc}
\toprule
Name & instances $n$ & features $d$ & nonzero \% \\
\midrule
STL-10 & 13000 & 27648 & 96.3 \\ 
CIFAR-10 & 60000 & 3073 & 99.7 \\
CIFAR-10-rf & 60000 & 60000 & 100.0\\
smallNorb-rf & 24300 & 30000 & 100.0 \\
E2006.train & 16087 & 150348 & 0.8 \\
sector & 6412 & 55197 & 0.3 \\
p53-rf & 16592 & 20000 & 100.0 \\
connect-4-rf & 16087 & 30000 & 100.0 \\
realsim-rf & 72309 & 50000 & 100.0  \\
rcv1-rf & 20242 & 30000 & 100.0 \\
cod-rna-rf & 59535 & 60000 & 100.0 \\
\bottomrule
\end{tabular}
\end{small}
\end{center}
\vskip -0.1in
\end{table}
In this section, we evaluate the performance of NysADMM
on different large-scale applications: lasso, 
$\ell_1$-regularized logistic regression, and SVM. 
For each type of problems, we compare NysADMM 
with popular standard solvers. We run all experiments on a server 
with 128 Intel Xeon E7-4850 v4 2.10GHz CPU cores and 1056GB. 
We repeat every numerical experiment ten times and report
the mean solution time. 
We highlight the best-performing method in bold. The tolerance of NysADMM at each iteration is chosen as the geometric mean $\varepsilon^{k + 1} = \sqrt{r_p^k r_d^k}$ of the ADMM primal residual $r_p$ and dual residual $r_d$ at the previous iteration, as in \cite{osqp}. See \citet{MAL-016} for more motivation and details. 
An alternative is to choose the tolerance sequence as any decaying sequence with respect to the righthand side norm as the number of NysADMM iteration increases, e.g., $\varepsilon^k = \|r^k\|_2 / k^{\beta}$, where $\beta$ is a predefined factor. These two strategies perform similarly; our experiments use the first strategy.

We choose a sketch size $s=50$ to compute the Nystr{\"o}m approximation throughout our experiments. Inspired by \Cref{thm:AdaNysADMM} and \Cref{corr:SmallRanks}, even if the sketch size is much smaller than the effective dimension, NysADMM can still achieve significant acceleration in practice. 

To support experiments with standard solvers, 
for each problem class we use the same stopping criterion
and other parameter settings as the standard solver.
These experiments use datasets with $n>10,000$ or $d>10,000$ from LIBSVM \cite{libsvm}, UCI \cite{uci}, and OpenML \cite{openml},
with statistics summarized in \Cref{dataset}.
We use a random feature map \cite{NIPS2007_013a006f, 4797607} to generate features 
for the data sets CIFAR-10, smallnorb, realsim, rcv1, and cod-rna, 
which increases both predictive performance and problem dimension.

\subsection{Lasso}
\label{subsection:expl1problem}
This subsection demonstrates the performance of NysADMM to solve the standard lasso problem \eqref{eq2}.
Here we compare NysADMM 
with three standard lasso solvers, SSNAL \cite{doi:10.1137/16M1097572}, mfIPM \cite{fountoulakis2014matrix}, and glmnet \cite{friedman2010regularization}.  
SSNAL is a Newton method based solver; mfIPM is an interior point method based solver and glmnet is a coordinate descent based solver. 
In practice, these three solvers and NysADMM rely on 
different stopping criteria. 
In order to make a fair comparison, in our experiments, 
the accuracy of a solution $x$ for \eqref{eq2} is measured by the following relative Karush–Kuhn–Tucker (KKT) residual \cite{doi:10.1137/16M1097572}: 
\begin{equation}\label{eq31}
	\eta = \frac{\|x - \text{prox}_{\gamma\|\cdot\|_1}(x - A^T(Ax - b))\|}{1 + \|x\| + \|Ax - b\|}.
\end{equation}
For a given tolerance $\epsilon$, we stop the tested algorithms when $\eta < \epsilon$.
Note that stopping criterion \eqref{eq31} is rather strong:
if $\eta \leq 10^{-2}$ for NysADMM, then
the primal and dual gaps for ADMM are $\lesssim 10^{-4}$, 
which suffices for most applications. 
Indeed, for many machine learning problems, 
lower bounds on the statistical performance of the estimator \citep{loh2017lower}
imply an unavoidable level of statistical error that is greater than this optimization error for most applications.
Optimizing the objective beyond the level of statistical error \citep{agarwal2012fast, loh2015regularized} 
does not improve generalization. 
For standard lasso experiments, we fix the regularization parameter at $\gamma = 1$. 
\begin{table}[tbhp]
\caption{Results for low precision lasso experiment.}	
\label{lassolowres}
\vskip 0.15in
\begin{center}
	\begin{small}
	\begin{tabular}{lcccc}
	\toprule
	\multirow{2}{*}{Task} & \multicolumn{4}{c}{Time for $\epsilon = 10^{-1}$ (s)} \\\cline{2-5}
	& NysADMM & mfIPM & SSNAL & glmnet \\
	\midrule
	STL-10 & \bf{165} &  573 & 467 & 278 \\ 
	CIFAR-10-rf & \bf{251} & 655 & 692 & 391  \\
	smallNorb-rf & \bf{219}  & 552  & 515 & 293  \\
	E2006.train &  \bf{313} &  875  & 903 & 554   \\
	sector & \bf{235} & 678  & 608  & 396  \\
	realsim-rf & \bf{193} &  -- & 765 & 292  \\
	rcv1-rf & \bf{226} & 563 & 595 & 273 \\
	cod-rna-rf & \bf{208} & 976 & 865 & 324 \\
	\bottomrule
	\end{tabular}
	\end{small}
\end{center}
\vskip -0.1in
\end{table}
\begin{table}[tbhp]
	\caption{Results for high precision lasso experiment.}	
	\label{lassohighres}
	\vskip 0.15in
	\begin{center}
	\begin{small}
	\begin{tabular}{lcccc}
	\toprule
	\multirow{2}{*}{Task} & \multicolumn{4}{c}{Time for $\epsilon = 10^{-2}$ (s)} \\\cline{2-5}
	& NysADMM & mfIPM & SSNAL & glmnet \\
	\midrule
	STL-10  & \bf{406} & 812 & 656 & 831\\ 
	CIFAR-10-rf & \bf{715} & 1317  & 1126 & 1169 \\
	smallNorb-rf & \bf{596} & 896 & 768 & 732 \\
	E2006.train & 1657 &  1965& \bf{1446} & 2135 \\
	sector & 957 & 1066 & \bf{875} & 1124 \\
	realsim-rf & \bf{732}& --& 1035 & 922 \\
	rcv1-rf & \bf{593}& 853 & 715 & 736 \\
	cod-rna-rf & \bf{715}& 1409 & 1167 & 997\\
	\bottomrule
	\end{tabular}
	\end{small}
	\end{center}
	\vskip -0.1in
\end{table}

\Cref{lassolowres} and \Cref{lassohighres} show results for lasso experiments. 
The average solution time for NysADMM, mfIPM, SSNAL, and glmnet 
with $\epsilon = 10^{-1}, 10^{-2}$ on different tasks are provided. 
Here mfIPM fails to solve the realsim-rf instance since it requires $n < d$. 
For precision of $\epsilon = 10^{-1}$, 
NysADMM is faster than all other solvers and at least 3 times faster than both mfIPM and SSNAL. 
For precision of $\epsilon = 10^{-2}$, 
NysADMM is still faster than all other solvers for all instances 
except E2006.train and sector.  
The results are fair since both SSNAL and mfIPM are second-order solvers 
and can reach high precision. NysADMM and glmnet 
are first-order solvers; they reach low precision quickly, but improve accuracy more slowly than a second order method. 
In practice, for large-scale machine learning problems, 
a low precision solution usually suffices, 
as decreasing optimization error beyond the statistical
noise in the problem does not improve generalization.
Further, our algorithm achieves bigger improvements 
on dense datasets compared with sparse datasets,
as the factors of the Nystr{\"o}m approximation are dense
even for sparse problems. 
\begin{figure}[tbhp]
	\centering
	\includegraphics[width = \columnwidth]{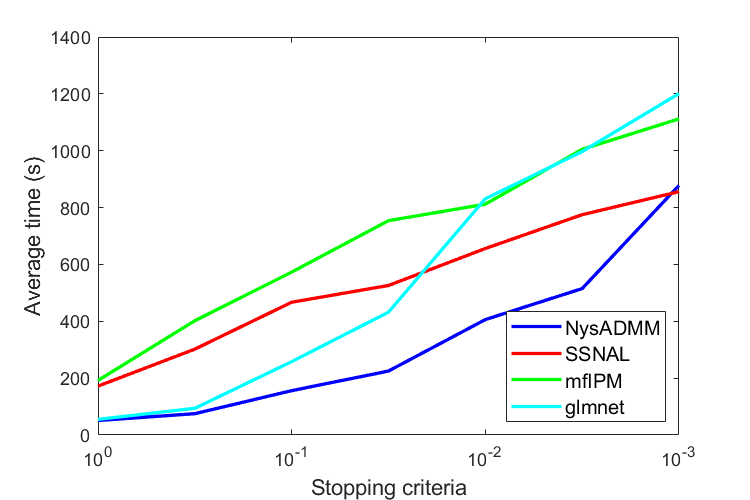}
	\caption{Solution times for varying tolerance $\epsilon$ on STL-10.}
	\label{precisionpath}
\end{figure}
To further illustrate the results, we vary the value of $\epsilon$ from 
$1.0$ to $10^{-3}$ on STL-10 task and plot the average solution time 
for four methods in \Cref{precisionpath}. 
We can see NysADMM is as least as fast as other solvers when
$\epsilon > 10^{-3}$,
and often twice as fast for many practical values of $\epsilon$. 

\subsection{$\ell_1$-regularized logistic regression}
\label{subsection:explogisticreg}
This subsection demonstrates the performance of NysADMM on 
$\ell_1$-regularized logistic regression, \eqref{eq9} from \Cref{subsection:logisticreg}. 
We test the method on binary classification problems 
using the same random feature map as in \Cref{subsection:expl1problem}.

The $\ell_1$-regularized logistic regression experiments compare NysADMM 
with the SAGA algorithm, a stochastic average gradient like algorithm \cite{NIPS2014_ede7e2b6} 
implemented in sklearn, and the accelerated proximal gradient (APG) algorithm \cite{beck2009fista, nesterov2013gradient, o2015adaptive}.
For the purpose of fair comparison, 
all the algorithms are stopped when the maximum relative change 
in the problem variable (that is, the regression coefficients) 
$\frac{\|x_k - x_{k+1}\|_\infty}{\|x_k\|_\infty}$ is less than the tolerance.
The tolerance is set to $10^{-3}$; other settings match
the default settings of the sklearn logistic regression solver. 

An overview of $\ell_1$-regularized logistic regression experiment results 
are provided in \Cref{logisticres}.
NysADMM uniformly out performs SAGA, solving each problem at least twice as fast. 
Similarly, NysADMM is at least twice as fast as APG on all datasets except STL-10, where it performs comparably.
In the cases of p53-rf and connect-4-rf, NysADMM runs significantly faster than its competitors, being four times faster than SAGA and three times faster than APG.
These large performance gains are due to the size of the problem instances and their conditioning. 
From \citep{NIPS2014_ede7e2b6}, the convergence speed of SAGA depends on the problem instance size and condition number. 
Our test cases have large instance sizes and condition numbers, which lead to slow convergence of SAGA. 
The situation with APG is similar. Indeed, although ADMM and proximal gradient methods generally have the same $O(1/t)$-convergence rate \cite{beck2009fista,he20121}, NysADMM is less sensitive ill-conditioning than APG. 

\begin{table}[tbhp]
	\caption{Results for $\ell_1$-regularized logistic regression experiment.}	
	\label{logisticres}
	\vskip 0.15in
	\begin{center}
	\begin{small}
	\begin{tabular}{lccc}
	\toprule
	Task & NysADMM (s) & SAGA (s) & APG (s) \\
	\midrule
	STL-10 & 3012 & 6083 & \bf{2635}\\ 
	CIFAR-10-rf & \bf{7884} & 21256 & 17292\\
	p53-rf & \bf{528} & 2116 & 1880\\
	connect-4-rf & \bf{866} & 4781 & 7365\\
	smallnorb-rf & \bf{1808} & 6381 & 4408\\
	rcv1-rf & \bf{1237} & 3988 & 2759\\
	con-rna-rf & \bf{7528} & 21513 & 16361 \\
	\bottomrule
	\end{tabular}
	\end{small}
	\end{center}
	\vskip -0.1in
\end{table}
\subsection{Support vector machine}
\label{subsection:expsvm}
This subsection demonstrates the performance of NysADMM 
on kernel SVM problem for binary classification, \eqref{eq8} from \Cref{subsection:svm}. 
The SVM experiments compare NysADMM with the LIBSVM solver \cite{libsvm}. 
LIBSVM uses sequential minimal optimization (SMO)
to solve the dual SVM problem. 
We use the same stopping criteria as the LIBSVM solver, 
which stops the NysADMM method when the ADMM dual gap reaches $10^{-4}$ level. 
All SVM experiments use the RBF kernel. 
\begin{table}[tbhp]
	\caption{Results of SVM experiment.}	
	\label{svmres}
	\vskip 0.15in
	\begin{center}
	\begin{small}
	\begin{tabular}{lcc}
	\toprule
	Task & NysADMM time (s) & LIBSVM time (s) \\
	\midrule
	STL-10 & \bf{208} & 11573 \\ 
	CIFAR-10 & \bf{1636} & 8563 \\
	p53-rf & \bf{291} & 919 \\
	connect-4-rf & \bf{7073} & 42762 \\
	realsim-rf & \bf{17045} & 52397 \\
	rcv1-rf & \bf{564} & 32848 \\
	cod-rna-rf & \bf{4942} & 36791 \\
	\bottomrule
	\end{tabular}
	\end{small}
	\end{center}
	\vskip -0.1in
\end{table}
\Cref{svmres} shows the results of SVM experiments.
On these problems, NysADMM is at least 3 times faster (and up to 58 times faster) than the LIBSVM solver. 
Consider problem formulation \eqref{eq8}, 
with the RBF kernel. The Gram matrix $\text{diag}(b) K \text{diag}(b)$ 
is dense and approximately low rank: exactly the setting in which 
NysADMM should be expected to perform well.
In constrast, the SMO-type decomposition in LIBSVM solver 
works better for sparse problems, as it updates only two variables at each iteration. 
\section{Conclusion}
\label{section:conclusion}
In this paper, we have developed a scalable new algorithm, NysADMM, that combines inexact ADMM and the randomized low-rank Nystr{\"o}m approximation to accelerate composite convex optimization. We show that NysADMM exhibits strong benefits both in theory and in practice. 
Our theory shows that when the Nystr{\"o}m preconditioner is constructed with an appropriate rank, NysADMM requires only a constant number of matvecs to solve the ADMM subproblem. 
We have also provided an adaptive strategy for selecting the rank that possesses a similar computational profile to the non-adaptive algorithm, and allows us to realize the theoretical benefits in practice. 
Further, numerical results demonstrate that NysADMM is as least twice as fast as standard methods on large dense lasso, regularized logistic regression, and kernalized SVM problems.
More broadly, this paper shows the promise of recent advances in RandNLA to 
provide practical accelerations for important large-scale optimization algorithms.

\bibliography{camera_ready.bib}
\bibliographystyle{icml2022}

\newpage
\appendix
\onecolumn
\section{Proofs of main results}
\label{section:proofs}
In this section we give the proofs for the main results of the paper: \cref{thm:NysCondNum}, \cref{thm:NysADMMConv}, and \cref{thm:AdaNysADMM}.
\subsection{Preliminaries}
We start by recalling some useful background information and technical results that are useful for proving the main theorems. In order to obtain the exponentially small failure probabilities in \cref{thm:NysCondNum} and \cref{thm:AdaNysADMM} we take a different approach from the one in \citet{frangella2021randomized}. 
The proofs are based on regularized Schur complements and approximate matrix multiplication. Our arguments are inspired by the techniques used to establish statistical guarantees for approximate kernel ridge regression via column sampling schemes \cite{bach2013sharp, alaoui2015fast}. 
\subsubsection{Nystr{\"o}m Approximation: Properties}
We start by recalling some important properties of the Nystr{\"o}m approximation \eqref{Nyssketch}. We shall also need the regularized Nystr{\"o}m approximation. Recall that $\Omega\in\mathbb{R}^{d\times s}$ denotes the test matrix from which we construct the Nystr{\"o}m approximation. 
Given $\sigma>0$, the regularized Nystr{\"o}m approximation with respect to $\Omega$ is defined as
\begin{equation}
\label{eq:RegNysAppx}
    H\langle \Omega \rangle_{\sigma} = (H\Omega)(\Omega^{T}H\Omega+\sigma I)^{-1}(H\Omega)^{T} .
\end{equation}
Furthermore, let $H = V\Lambda V^{T}$ be the eigendecomposition of $H$ and define $D_{\sigma} = H(H+\sigma I)^{-1} = \Lambda(\Lambda+\sigma I)^{-1}$. We shall see below that $D_{\sigma}$ plays a crucial role in the analysis. 
The following lemmas are well known in the literature and summarize the properties of the Nystr{\"o}m and regularized Nystr{\"o}m approximation. \Cref{NysPropLemma} may be found in \citet{frangella2021randomized} and \cref{RegNysLemma} in \citet{alaoui2015fast}.
\begin{lemma}
\label{NysPropLemma}
Let $H\langle \Omega \rangle$ be a Nystr{\"o}m approximation of a symmetric psd matrix $H$.
Then 
\begin{enumerate}
    \item The approximation $H\langle \Omega\rangle $ is psd and has rank at most $s$.
    \item The approximation $H\langle \Omega \rangle$ depends only on $\mathrm{range}(\Omega)$.
    \item\label{it:nys-order} In the Loewner order, $H\langle \Omega\rangle\preceq H$.
    \item\label{it:nys-eigs} In particular, the eigenvalues satisfy $\lambda_j(H\langle \Omega \rangle) \leq \lambda_{j}(H)$ for each $1 \leq j \leq d$.
\end{enumerate}
\end{lemma}
\begin{lemma}
\label{RegNysLemma}
Let $H$ be a symmetric psd matrix, $\sigma>0$. Define $E = H-H\langle \Omega \rangle$ and $E_{\sigma} = H-H\langle \Omega \rangle_{\sigma}$. Then the following hold.
\begin{enumerate}
    \item $H\langle\Omega\rangle_{\sigma}\preceq H\langle\Omega\rangle\preceq H$.
    \item $0\preceq E \preceq E_{\sigma}.$
    \item If $\|D^{1/2}_{\sigma}V^{T}(\frac{1}{s}\Omega\Omega^{T})VD^{1/2}_{\sigma}-D_{\sigma}\|\leq \eta < 1$, then
    \begin{equation}
       0\preceq E_{\sigma} \preceq \frac{\sigma}{1-\eta}I. 
    \end{equation}
\end{enumerate}
\end{lemma}
\Cref{RegNysLemma} relates $H\langle \Omega\rangle_{\sigma}$ to $H\langle \Omega \rangle$ and $H$. In particular, item 2 implies that $\|E\|\leq \|E_{\sigma}\|$, so controlling $E_{\sigma}$ controls $E$. Item 3 shows that $E_{\sigma}$ can be controlled by the spectral norm of the matrix
\begin{equation}
\label{eq:AppxMatMul}
    D^{1/2}_{\sigma}V^{T}\frac{1}{s}\Omega\Omega^{T}VD^{1/2}_{\sigma}-D_{\sigma}.
\end{equation}
The spectral norm of \eqref{eq:AppxMatMul} can be bounded by observing 
\begin{align}
\mathbb{E}\left[D^{1/2}_{\sigma}V^{T}\frac{1}{s}\Omega\Omega^{T}VD^{1/2}_{\sigma}\right] &=  \\
D^{1/2}_{\sigma}V^{T}\mathbb{E}\left[\frac{1}{s}\Omega\Omega^{T}\right]VD^{1/2}_{\sigma} &= \\ 
D^{1/2}_{\sigma}V^{T}VD^{1/2}_{\sigma} = D_{\sigma}.\    
\end{align}
Thus, $D^{1/2}_{\sigma}V^{T}\frac{1}{s}\Omega\Omega^{T}VD^{1/2}_{\sigma}$ is an unbiased estimator of $D_{\sigma},$ and may be viewed as approximating the product of the matrices $D_{\sigma}^{1/2}V^{T}$ and $VD_{\sigma}^{1/2}$. 
Hence results from randomized linear algebra can bound the spectral norm of this difference. 
In particular, it suffices to take a sketch size that scales with the effective dimension, using results on approximate matrix multiplication in terms of stable rank \cite{cohen2016optimal}.    

\subsubsection{Approximate matrix multiplication in terms of the effective dimension}
The condition in item 3 of \cref{RegNysLemma} follows immediately from theorem 1 of \citet{cohen2016optimal}. Unfortunately, the analysis in that paper does not yield explicit constants.
Instead we use a special case of their results due to \citet{lacotte2021fast} that provides explicit constants. 
\Cref{thm:AppxMatMulEffDim}
 simplifies theorem 5.2 in \citet{lacotte2021fast}.
\begin{theorem}
\label{thm:AppxMatMulEffDim}
Let $\Psi \in \mathbb{R}^{s\times d}$ be a matrix with i.i.d. $N(0,\frac{1}{s})$ entries. Given $\delta>0$, and $\tau \in (0,1)$ it holds with probability at least $1-\delta$ that
\begin{align}
    & \sup_{v\in \mathbb{S}^{d-1}}\langle v,(D^{1/2}_{\sigma}V^{T}\Psi^{T}\Psi VD^{1/2}_{\sigma}-D_{\sigma})v\rangle \leq \tau +2\sqrt{\tau}, \\
    & \inf_{v\in \mathbb{S}^{d-1}}\langle v,(D^{1/2}_{\sigma}V^{T}\Psi^{T}\Psi VD^{1/2}_{\sigma}-D_{\sigma})v\rangle \geq \tau-2\sqrt{\tau},
\end{align}
provided $s\geq \frac{\left(\sqrt{d_{\textup{eff}}(\sigma)}+\sqrt{8\log(16/\delta)}\right)^{2}}{\tau}$.
\end{theorem}
Setting $\Psi = \frac{1}{\sqrt{s}}\Omega^{T}$, where $\Omega\in \mathbb{R}^{d\times s}$ has i.i.d. $N(0,1)$ entries, \cref{thm:AppxMatMulEffDim} yields the following corollary.
\begin{corollary}
\label{corr:AppxMatMul}
Let $\Omega \in \mathbb{R}^{d\times s}$ be a matrix with i.i.d. $N(0,1)$ entries. Given $\delta>0$, and $\tau \in (0,1)$ it holds with probability at least $1-\delta$ that
\begin{equation}
    \bigg\|D^{1/2}_{\sigma}V^{T}\frac{1}{s}\Omega\Omega^{T}VD^{1/2}_{\sigma}-D_{\sigma}\bigg\|\leq \tau+2\sqrt{\tau}
\end{equation}
provided $s\geq \frac{\left(\sqrt{d_{\textup{eff}}(\rho))}+\sqrt{8\log(16/\delta)}\right)^{2}}{\tau}$.
\end{corollary}

\subsubsection{Condition number of Nystr{\"o}m preconditoned linear system}
The following result is a simpler version of proposition 5.2 in \citet{frangella2021randomized}. 
\begin{proposition}
\label{NysPCGThm}
Let $\hat{H} = U \hat{\Lambda} U^T$ be any rank-$s$ Nystr{\"o}m approximation, with $s$th largest eigenvalue $\hat{\lambda}_{s}$,
and let $E = H-\hat{H}$ be the approximation error.  Construct the Nystr{\"o}m preconditioner $P$ as in \eqref{Nyspreconditioner}.
Then the condition number of the preconditioned matrix $P^{-1/2}H_{\rho}P^{-1/2}$ satisfies
\begin{equation}
    \label{eq-CondNumBnd}
\begin{aligned}
     \kappa_{2}(P^{-1/2}H_{\rho}P^{-1/2}) \leq \frac{\hat{\lambda}_{s}+\rho+\|E\|}{\rho}.
\end{aligned}
\end{equation}
\end{proposition}
\Cref{NysPCGThm} bounds the condition condition number of the Nystr{\"o}m preconditioned linear system in terms of $\hat{\lambda}_{s}, \rho$ and the approximation error $\|E\|$. 
We would like to emphasize that the bound in \cref{NysPCGThm} is deterministic.  

\subsection{Proofs of \Cref{thm:NysCondNum} and \Cref{corr:ADMMSubProb}}
We start with two lemmas from which \cref{thm:NysCondNum} follows easily. The first lemma and its proof appear in \citet{frangella2021randomized}.

\begin{lemma}
\label{KeyLemma}
Let $H\in \mathbb{S}_{n}^+(\mathbb{R})$ with eigenvalues $\lambda_{1}\geq\lambda_{2}\geq\cdots\geq\lambda_{d}$. Let $\rho > 0$ be regularization parameter,
and define the effective dimension as in~\eqref{eqn:deff-proof}.  Then the following statement holds.

Fix $\gamma>0$. If $j\geq (1+\gamma^{-1})d_{\textup{eff}}(\rho)$, then $\lambda_{j}\leq \gamma\rho$.

\end{lemma}

\begin{lemma}
\label{lemma:NysErrLemma}
Let $\epsilon>0$ and $E = H-H\langle \Omega \rangle$. Suppose we construct a randomized Nystr{\"o}m approximation from a standard Gaussian random matrix $\Omega$ with sketch size $s\geq 8\left(\sqrt{d_{\textup{eff}}(\epsilon)}+\sqrt{8\log(\frac{16}{\delta}})\right)^{2}$. Then the event 
\begin{equation}
    \mathcal{E} = \{\|E\|\leq 6\epsilon\},
\end{equation}
holds with probability at least $1-\delta$.
\end{lemma}
\begin{proof}
Let $\Omega_{s} = \frac{1}{\sqrt{s}}\Omega$ and observe that $H\langle \Omega_{s} \rangle = H\langle \Omega \rangle$. Now the conditions of \cref{corr:AppxMatMul} are satisfied with $\sigma = \epsilon$ and $\tau = 8$. Consequently with probability at least $1-\delta$,
\[ \bigg\|D^{1/2}_{\epsilon}V^{T}\frac{1}{s}\Omega\Omega^{T}VD^{1/2}_{\epsilon}-D_{\epsilon}\bigg\|\leq \frac{1}{8}+\frac{\sqrt{2}}{2}.\]
Hence applying \cref{RegNysLemma} with $\sigma = \epsilon$ and $\eta = \frac{1}{8}+\frac{\sqrt{2}}{2}$, we obtain
\[\bigg\|H-H\langle \Omega_{s}\rangle_{\epsilon}\bigg\|\leq 6\epsilon,\]
with probability at least $1-\delta$.
Recalling our initial observation, we conclude the desired result.
\end{proof}
\subsubsection{Proof of \cref{thm:NysCondNum}}
\begin{proof}
As $s\geq 8\left(\sqrt{d_{\textrm{eff}}(\rho)}+\sqrt{8\log(\frac{16}{\delta}})\right)^{2}$ we have that $\|E\|\leq 6\rho$ with probability at least $1-\delta$ by \cref{lemma:NysErrLemma}.
Furthermore, $\hat{\lambda}_{s}\leq \frac{\rho}{7}$ by item 3 of \cref{NysPropLemma} and \cref{KeyLemma} with $\gamma = 1/7$. 
Combining this with \cref{NysPCGThm}, we conclude with probability at least $1-\delta$,
\begin{align*}
    \kappa_{2}(P^{-1/2}H_{\rho}P^{-1/2})&\leq \frac{\hat{\lambda}_{s}+\rho+\|E\|}{\rho} \\
    &\leq 1+6+\frac{1}{7}\leq 8
\end{align*}
as desired.
\end{proof}

\subsubsection{Proof of \Cref{corr:ADMMSubProb}}
\begin{proof}
Let $A = P^{-1/2}H_{\rho}P^{-1/2}$ and condition on the event that $\kappa_{2}(A)\leq 8$, which holds with probability at least $1-\delta$. The standard theory for convergence of CG \cite{trefethen1997numerical} guarantees after $t$ iterations that,
\begin{equation}
    \frac{\|x_t-\tilde{x}_{\star}\|_{A}}{\|\tilde{x}_\star\|_{A}} \leq 2\left(\frac{\sqrt{\kappa_2(A)}-1}{\sqrt{\kappa_{2}(A)}+1}\right)^{t}
\end{equation}
where $\|x\|_{A} = x^{T}Ax$.
\Cref{thm:NysCondNum} guarantees that the \nystrom preconditioned matrix satisfies $\kappa_{2}(A)\leq 8$, so the above display may be majorized as
\begin{equation}
    \frac{\|x_t-\tilde{x}_{\star}\|_{A}}{\|\tilde{x}_\star\|_{A}} \leq \left(\frac{1}{2}\right)^{t-1}.
\end{equation}
Now, from the elementary inequality 
\begin{equation}
\lambda_{d}(A)\|x\|_2\leq \|x\|_{A}\leq \lambda_{1}(A)\|x\|_{2}, 
\end{equation}
we conclude
\begin{equation}
    \frac{\|x_t-\tilde{x}_{\star}\|_{2}}{\|\tilde{x}_\star\|_{2}}\leq \kappa_{2}(A)\left(\frac{1}{2}\right)^{t-1}\leq \left(\frac{1}{2}\right)^{t-4}.
\end{equation}
To obtain the claimed result, multiply both sides by $\|\tilde{x}_\star\|_2$ and solve $\|\tilde{x}_\star\|_2\left(\frac{1}{2}\right)^{t-4} = \epsilon$ for $t$. 
\end{proof}

\subsection{Proof of \Cref{thm:NysADMMConv}}
This proof is a natural consequence of the following theorem from \citet{eckstein1992douglas}. 

\begin{theorem}
\label{thm:inexactadmmconvergence}
Consider a convex optimization problem in the primal form (P), $\textup{minimize} \; f(x) + h(Mx)$, where $x \in \mathbb{R}^d$, $M \in \mathbb{R}^{m \times d}$ has full column rank. Pick any $y^0$, $z^0 \in \mathbb{R}^m$, and $\rho >0 $, and summable sequences
\begin{equation*}
    \begin{aligned}
    & \{\varepsilon^k\}_{k = 0}^{\infty} \subseteq [0, \infty), \; \sum^{\infty}_{k = 0} \varepsilon^k < \infty,\\ & \{\nu^k\}_{k = 0}^{\infty} \subseteq [0, \infty), \; \sum^{\infty}_{k = 0} \nu^k < \infty, \\
    & \{\lambda^k\}_{k = 0}^{\infty} \subseteq (0, 2), \; 0 < \inf \lambda^k \le \sup \lambda^k < 2. \\
    \end{aligned}
\end{equation*}
The dual problem (D) of primal problem (P) is
\begin{equation*}
    \textup{maximize}_{y \in \mathbb{R}^m} \; -(f^*(-M^Ty) + g^*(y)). 
\end{equation*}
Suppose the primal and dual ADMM iterates $\{x^k\}_{k = 0}^{\infty}$, $\{z^k\}_{k = 0}^{\infty}$, and $\{y^k\}_{k = 0}^{\infty}$ satisfy the update equations to within errors given by  conform, for all $k$ to 
\begin{equation}
\label{summablecondition}
    \begin{aligned}
    & \bigg\|x^{k + 1} - \textup{argmin}_x \big\{ f(x) + \langle y^k, Mx \rangle \\ 
    & + \frac{1}{2} \rho \|Mx - z^k\|_2^2\big\}\bigg\|_2 \le \varepsilon^k, \\
    & \bigg\|z^{k + 1} - \textup{argmin}_z \big\{ h(z) - \langle y^k, z \rangle  \\ 
    & + \frac{1}{2} \rho \|\lambda^k Mx^{k + 1} - z + (1 - \lambda^k) z^k \|_2^2\big\}\bigg\|_2 \le \nu^k,\\
    & y^{k + 1} = y^k + \rho (\lambda^k Mx^{k + 1} + (1 - \lambda^k) z^k - z^{k + 1}).
    \end{aligned}
\end{equation}
Then if (P) has a Kuhn-Tucker pair, $\{x^k\}$ converges to a solution of (P) and $\{y^k\}$ converges to a solution of (D).  
\end{theorem}
\subsubsection{Proof of \cref{thm:NysADMMConv}}
\begin{proof}
Consider optimization problem \eqref{eq1} and the associated NysADMM algorithm \cref{NysADMM}.
Suppose $\{\tilde x^k\}_{k = 0}^{\infty}$, $\{\tilde z^k\}_{k = 0}^{\infty}$, and $\{\tilde u^k\}_{k = 0}^{\infty}$ are generated by NysADMM iterations. 
Since $\ell(Ax, b)$ is quadratic with respect to $x$ and the smooth part $g$ of regularizer $r$ has constant Hessian, the $x$-subproblem of \eqref{eq1} is exactly the linear system \eqref{ADMMSubProb}. 

Let $x^{k+1}$ be the exact solution for the $x$-subproblem at iteration $k$. For all $k \ge 0$, NysADMM iterate $\tilde x^{k+1}$ satisfies $\|\tilde x^{k+1} - x^{k+1}\|_2 \le \varepsilon^k$. Let $M = I$, $\nu^k = 0$, $\lambda^k = 1$, $y^k = \rho \tilde u^k$ for all $k$, and $f(x) = \ell(Ax, b) + g(x)$.
By \cref{thm:inexactadmmconvergence}, $\{\tilde x^k\}_{k = 0}^{\infty}$, $\{\tilde z^k\}_{k = 0}^{\infty}$, and $\{\rho \tilde u^k\}_{k = 0}^{\infty}$ satisfy condition \eqref{summablecondition}. Therefore, if optimization problem \eqref{eq1} has a Kuhn-Tucker pair, $\{\tilde x^k\}$ converges to a solution of \eqref{eq1} and $\{\rho \tilde u^k\}$ converges to a solution of the dual problem of \eqref{eq1}.  

Next, we derive the bound for the number of Nystr{\"o}m PCG iterations $T^k$ required at NysADMM iteration $k$. 
Note that in this case the Hessians of $\ell$ and $g$ are constant. We only need to sketch once for the constant linear system matrix $A^TH^{\ell}(A \tilde x^k ; b)A + H^g( \tilde x^k)$ and can reuse the sketch for all NysADMM iterations. Since the Nystr{\"o}m preconditioner is constructed with sketch size $s \ge 8\left(\sqrt{d_{\textrm{eff}}(\rho)}+\sqrt{8\log(\frac{16}{\delta}})\right)^{2}$, by \cref{corr:ADMMSubProb}, with probability at least $1 - \delta$, after 
\begin{equation*}T^k \ge \bigg\lceil\frac{\log\left(\frac{16\|x^{k+1}\|_2}{\varepsilon^k}\right)}{\log(2)}  \bigg\rceil
\end{equation*} 
Nystr{\"o}m PCG iterations, we have $\|\tilde x^{k+1} - x^{k+1}\|_2 \le \varepsilon^k$. 
Recall the righthand side of linear system $\eqref{ADMMSubProb}$ $r^k$. The exact solution for the $x$-subproblem $x^{k + 1}$ at iteration $k$ satisfies $\|x^{k + 1}\|_2 \le \frac{\|r^k\|_2}{\rho}$.
We have 
\begin{equation*}
    \bigg\lceil\frac{\log\left(\frac{16\|x^{k+1}\|_2}{\varepsilon^k}\right)}{\log(2)}  \bigg\rceil \le \bigg\lceil\frac{\log\left(\frac{16\|r^k\|_2}{\varepsilon^k \rho}\right)}{\log(2)} \bigg\rceil.
\end{equation*}
Further, by assumption, as $\|r^k\|_2$ is bounded by a constant $R$ for all $k$, we have 
\begin{equation*}
 \bigg\lceil\frac{\log\left(\frac{16\|r^k\|_2}{\varepsilon^k \rho}\right)}{\log(2)} \bigg\rceil \le 4 +  \bigg\lceil\frac{\log\left(\frac{R}{\varepsilon^k \rho}\right)}{\log(2)} \bigg\rceil \le 4 +  \bigg\lceil2\log\left(\frac{R}{\varepsilon^k \rho}\right) \bigg\rceil. 
\end{equation*}
This gives the bound for the number of Nystr{\"o}m PCG iterations $T^k$ required at NysADMM iteration $k$
\end{proof}

\subsection{Proof of \Cref{thm:AdaNysADMM}}
\begin{proof}
By  hypothesis we have $s>8d_{\textup{eff}}(\frac{\epsilon\rho}{6})$, so \cref{KeyLemma} with $\gamma = 7$ yields 
\begin{align*}
\hat{\lambda}_s\leq \lambda_s & \leq \frac{1}{7}\frac{\epsilon\rho}{6} = \frac{\epsilon \rho}{42},
\end{align*}
Thus,
\[\frac{\hat{\lambda}_s+\rho}{\rho}\leq 1+\frac{\epsilon}{42}.\]
This gives the first statement.  
For the second statement we use our hypothesis on $s$ to apply \cref{lemma:NysErrLemma} with tolerance $\epsilon \rho/6$. From this we conclude $\|E\|\leq \epsilon \rho$ with probability at least $1-\delta$. Combining this with \cref{NysPCGThm} yields 
\[\kappa_{2}(P^{-1/2}H_{\rho}P^{-1/2})- \frac{\hat{\lambda}_s+\rho}{\rho}\leq \epsilon,\]
with probability at least $1-\delta$.
On the other hand, condition numbers always satisfy
\[\kappa_{2}(P^{-1/2}H_{\rho}P^{-1/2})\geq 1.\] 
Combining this with our upper bound on $\hat{\lambda}_{s}$ gives
\begin{align*}
    \kappa_{2}(P^{-1/2}H_{\rho}P^{-1/2})-\frac{\hat{\lambda}_s+\rho}{\rho}&\geq 1-(1+\epsilon/42)\\ 
    &= -\epsilon/42.
\end{align*}
Hence with probability at least $1-\delta$
\[\bigg|\kappa_{2}(P^{-1/2}H_{\rho}P^{-1/2})-\frac{\hat{\lambda}_s+\rho}{\rho}\bigg|\leq \epsilon.\]
\end{proof}

\section{Randomized Nystr{\"o}m approximation and Nystr{\"o}m PCG}
In this section we give the algorithms from \citet{frangella2021randomized} for the randomized Nystr{\"o}m approximation and Nystr{\"o}m PCG.
\label{section:AlgNysAppxNysPCG}
\begin{algorithm}[H]
	\centering
	\caption{Randomized Nystr{\"o}m Approximation}
	\label{Nyssketch}
	\begin{algorithmic}
		\INPUT{psd matrix $H \in \mathbb{S}_d^+(\mathbb{R})$, sketch size $s$}
		\STATE{$\Omega = \text{randn}(d, s)$} \COMMENT{Gaussian test matrix}
		\STATE{$\Omega = \text{qr}(\Omega, 0)$} \COMMENT{thin QR decomposition}
		\STATE{$Y = H\Omega$} \COMMENT{$s$ matvecs with $H$}
		\STATE{$\nu = \text{eps}(\text{norm}(Y, 2))$} \COMMENT{compute shift}
		\STATE{$Y_{\nu} = Y + \nu \Omega$} \COMMENT{add shift for stability}
		\STATE{$C = \text{chol}(\Omega^TY_{\nu})$} \COMMENT{Cholesky decomposition}
		\STATE{$B = Y_{\nu} / C$} \COMMENT{triangular solve}
		\STATE{$[U, \Sigma, \sim] = \text{svd}(B, 0)$} \COMMENT{thin SVD}
		\STATE{$\hat \Lambda = \text{max}\{0, \Sigma^2 - \nu I\}$} \COMMENT{remove shift, compute eigs}
		\OUTPUT{Nystr{\"o}m approximation $\hat H_{\text{nys}} = U \hat \Lambda U^T$}
	\end{algorithmic}
\end{algorithm} 

\begin{algorithm}[t]
	\centering
	\caption{Nystr{\"o}m PCG}
	\label{Nyspcg}
	\begin{algorithmic}
		\INPUT{psd matrix $H$, righthand side $r$, initial guess $x_0$, regularization parameter $\rho$, sketch size $s$, tolerance $\varepsilon$}
		\STATE{$[U, \hat \Lambda] = \text{RandomizedNystr{\"o}mApproximation}(H,s)$}
		\STATE{$w_0 = r - (H + \rho I)x_0$}
		\STATE{$y_0 = P^{-1}w_0$}
		\STATE{$p_0 = y_0$}
		\WHILE{$\|w\|_2 > \varepsilon$}
		\STATE{$v= (H + \rho I)p_0$} 
		\STATE{$\alpha = (w_0^T y_0) / (p_0^T v)$}
		\STATE{$x = x_0 + \alpha p_0$}
		\STATE{$w = w_0 - \alpha v$}
		\STATE{$y = P^{-1}w$}
		\STATE{$\beta = (w^T y) / (w_0^T y_0)$}
		\STATE{$x_0 \gets x$, $w_0 \gets w$, $p_0 \gets y + \beta p_0$, $y_0 \gets y$}
		\ENDWHILE
		\OUTPUT{approximate solution $\hat x$}
	\end{algorithmic}
\end{algorithm} 

\section{AdaNysADMM}
In this section we give the adaptive algorithm for computing the randomized Nystr{\"o}m approximation adopted from \citet{frangella2021randomized}. 
The adaptive algorithm has the benefit of reusing computation, in particular, we do not need to compute the sketch $Y$ from scratch. 
We simply add onto the sketch that we have already computed.
We also give the pseudo-code for AdaNysADMM that uses \cref{alg:AdaRandNysAppx} to compute the Nystr{\"o}m preconditioner.  
\begin{algorithm}[H]
\centering
\caption{AdaptiveRandNysAppx}
\label{alg:AdaRandNysAppx}
\begin{algorithmic}
    \INPUT{symmetric psd matrix $H$, initial rank $s_0$, tolerance $\textrm{Tol}$}
    \STATE{$Y = [\hspace{5pt}], \Omega = [\hspace{5pt}],$ and $(\hat{\lambda}_s+\rho)/\rho = \textrm{Inf}$}
    \STATE{$m = s_{0}$}
    \WHILE{($\hat{\lambda}_s+\rho)/\rho>\textrm{Tol}$}
        \STATE{generate Gaussian test matrix $\Omega_{0}\in \mathbb{R}^{n\times m}$}
		\STATE{$[\Omega_{0},\sim] = \textrm{qr}(\Omega_{0},0)$}
		\STATE{$Y_{0} = H\Omega_{0}$}
		\STATE{$\Omega = [\Omega\hspace{5pt}\Omega_{0}]$ and $Y = [Y\hspace{5pt}Y_{0}]$}
		\STATE{$\nu = \sqrt{n}\hspace{2pt}\textrm{eps}(\textrm{norm}(Y,2))$}
		\STATE{$Y_{\nu} = Y+\nu\Omega,\hspace{3pt}$}
		\STATE{$C = \textrm{chol}(\Omega^{T}Y_{\nu})$}
		\STATE{$B = Y_{\nu}/C$}
		\STATE{compute $[U,\Sigma,\sim] = \textrm{svd}(B,0)$}
		\STATE{$\hat{\Lambda} = \max\{0,\Sigma^{2}-\nu I\}$} \COMMENT{remove shift}
		\STATE{compute $(\hat{\lambda}_s+\rho)/\rho$}
		\STATE{$m\leftarrow s_{0}$, $s_{0} \leftarrow 2s_{0}$} \COMMENT{double rank if tolerance is not met}
		    \IF{$s_{0}>s_{\textrm{max}}$}
		    \STATE{$s_{0} = s_{0}-m$}  \COMMENT{when $s_{0}>s_{\textrm{max}}$, reset to $s_{0} = s_{\textrm{max}}$}
		    \STATE{$m = s_{\textrm{max}}-s_{0}$}
		    \STATE{generate Gaussian test matrix $\Omega_{0}\in \mathbb{R}^{n\times m}$}
		    \STATE{$[\Omega_{0},\sim] = \textrm{qr}(\Omega_{0},0)$}
		    \STATE{$Y_{0} = H\Omega_{0}$}
		    \STATE{$\Omega = [\Omega\hspace{5pt}\Omega_{0}]$ and $Y = [Y\hspace{5pt}Y_{0}]$}
		    \STATE{$\nu = \sqrt{n}\hspace{2pt}\textrm{eps}(\textrm{norm}(Y,2))$} \COMMENT{compute final approximation and break}
		    \STATE{$Y_{\nu} = Y+\nu\Omega,\hspace{3pt}$}
		    \STATE{$C = \textrm{chol}(\Omega^{T}Y_{\nu})$
		                \STATE $B = Y_{\nu}/C$}
		    \STATE{compute $[U,\Sigma,\sim] = \textrm{svd}(B,0)$}
		    \STATE{$\hat{\Lambda} = \max\{0,\Sigma^{2}-\nu I\}$}
		    \STATE {\textbf{break}}
		    \ENDIF
        \ENDWHILE
    \OUTPUT{Nystr{\"o}m approximation $(U,\hat{\Lambda})$}
\end{algorithmic}
\end{algorithm}

\begin{algorithm}[H]
	\centering
	\caption{AdaNysADMM}
	\label{alg:AdaNysADMM}
	\begin{algorithmic}
	\INPUT{feature matrix $A$, response $b$, loss function $\ell$, regularization $g$ and $h$, stepsize $\rho$, positive summable sequence $\{\varepsilon^k\}_{k = 0}^{\infty}$}
	    \STATE{$[U, \hat \Lambda] =  \text{AdaptiveRandNysAppx}(A^T H^\ell A + H^g,s)$} \COMMENT{use \cref{alg:AdaRandNysAppx}}
		\REPEAT
		\STATE{find $\tilde x^{k+1}$ that solves \eqref{ADMMSubProb} within tolerance $\varepsilon^k$ by Nystr{\"o}m PCG}
		\STATE{$\tilde z^{k+1} = \text{argmin}_z \{h(z) + \frac{\rho}{2}\|\tilde x^{k+1} - z + \tilde u^k\|^2_2\}$}
		\STATE{$\tilde u^{k+1} = \tilde u^k + \tilde x^{k+1} - \tilde z^{k+1} $}
		\UNTIL convergence
	\end{algorithmic}
\end{algorithm}


\end{document}